\documentclass[10pt,reqno]{amsart}

\usepackage{a4wide}
\usepackage{dsfont}
\usepackage{amssymb,amsmath,amsfonts}
\usepackage{mathrsfs}
\usepackage[all]{xy}

\newtheorem{proposition}{Proposition}[section]
  \newtheorem{theorem}[proposition]{Theorem}
  \newtheorem{corollary}[proposition]{Corollary}
  \newtheorem{lemma}[proposition]{Lemma}
\theoremstyle{definition}
  
  \newtheorem{remark}[proposition]{Remark}

\newcommand{\cst}{\ifmmode\mathrm{C}^*\else{$\mathrm{C}^*$}\fi}
\newcommand{\st}{\;\vline\;}
\newcommand{\CC}{\mathbb{C}}
\newcommand{\NN}{\mathbb{N}}
\newcommand{\ZZ}{\mathbb{Z}}
\newcommand{\GG}{\mathbb{G}}
\newcommand{\RR}{\mathbb{R}}

\newcommand{\is}[2]{\left\langle#1\,\vline\,#2\right\rangle}
\newcommand{\tens}{\otimes}
\newcommand{\id}{\mathrm{id}}
\newcommand{\comp}{\!\circ\!}
\newcommand{\eps}{\varepsilon}
\newcommand{\vtens}{\,\bar{\otimes}\,}
\newcommand{\I}{\mathds{1}}
\newcommand{\cH}{\mathscr{H}}
\newcommand{\cK}{\mathscr{K}}
\newcommand{\sA}{\mathsf{A}}
\newcommand{\sB}{\mathsf{B}}
\newcommand{\cT}{\mathcal{T}}
\newcommand{\cU}{\mathcal{U}}

\newcommand{\bh}{\boldsymbol{h}}
\newcommand{\bpi}{\boldsymbol{\pi}}
\newcommand{\balpha}{\boldsymbol{\alpha}}
\newcommand{\bgamma}{\boldsymbol{\gamma}}
\newcommand{\bu}{\boldsymbol{u}}
\newcommand{\bs}{\boldsymbol{s}}
\newcommand{\bN}{\boldsymbol{N}}

\DeclareMathOperator{\C}{C}
\DeclareMathOperator{\B}{B}
\DeclareMathOperator{\Sp}{Sp}
\DeclareMathOperator{\Pol}{Pol}
\DeclareMathOperator{\Linf}{\mathnormal{L}^\infty\;\!\!}

\DeclareMathOperator{\Phase}{Phase}
\DeclareMathOperator{\vN}{vN}
\DeclareMathOperator{\sgn}{sgn}

\numberwithin{equation}{section}

\author{Jacek Krajczok}
\address{Department of Mathematical Methods in Physics, Faculty of Physics, University of Warsaw}
\email{jk347906@okwf.fuw.edu.pl}

\author{Piotr M.~So{\l}tan}
\thanks{Partially supported by the National Science Center (NCN) grant no.~2015/17/B/ST1/00085}
\address{Department of Mathematical Methods in Physics, Faculty of Physics, University of Warsaw}
\email{piotr.soltan@fuw.edu.pl}

\title[Center of $\operatorname{C}(\mathrm{SU}_q(2))$]{Center of the algebra of functions on the quantum group $\mathrm{SU}_q(2)$ and related topics}



\begin{document}

\begin{abstract}
The center of the algebra of continuous functions on the quantum group $\mathrm{SU}_q(2)$ is determined as well as centers of other related algebras. Several other results concerning this quantum group are given with direct proofs based on concrete realization of these algebras as algebras of operators on a Hilbert space.
\end{abstract}

\maketitle

\begin{center}
\emph{Dedicated to Marek Bo\.zejko on the occasion of his 70th birthday.}
\end{center}

\section{Introduction}

The aim of this paper is to provide very direct and relatively elementary proofs of certain facts concerning the quantum $\mathrm{SU}(2)$ group introduced by S.L.~Woronowicz in the seminal paper \cite{su2}. The issues addressed in this paper are the following
\begin{itemize}
\item faithfulness of the representation $\bpi$ introduced in \cite[Proof of Theorem 1.2]{su2},
\item\sloppy determining the center of the algebras $\Pol(\mathrm{SU}_q(2))$, $\C(\mathrm{SU}_q(2))$ and $\Linf(\mathrm{SU}_q(2))$ as well as the commutant of $\bpi(\C(\mathrm{SU}_q(2)))$,
\item giving direct proofs of faithfulness of Haar measure and continuity of the counit.
\end{itemize}
The above tasks are interrelated and the relations between them will be explained in detail.

Most results of this work are taken from the first author's BSc thesis submitted at the Faculty of Physics, University of Warsaw. These results are known and in most cases proofs are published, but our approach is rather elementary and direct. 

The quantum groups $\mathrm{SU}_q(2)$ were introduced in \cite{su2} and later studied in numerous papers in mathematics and theoretical physics. Apart from \cite{su2} our approach will be based on fundamental texts \cite{pseudogroups,cqg} and more specialized \cite{bmt,lance,NeshTu,contraction}. Methods of functional analysis and operator algebras are covered in textbooks such as \cite{davidson,rudin,sz}.

The paper is organized as follows: in the next subsections we briefly introduce terminology and notation needed in the remainder of the paper. Section \ref{Asect} is devoted to a detailed proof of faithfulness of a particular representation of the algebra of functions on the quantum group $\mathrm{SU}_q(2)$ defined in \cite{su2}. In Section \ref{qg} we introduce the additional structure on the \cst-algebra studied in Section \ref{Asect} which defines the quantum group $\mathrm{SU}_q(2)$. We also list some objects needed for later sections and recall the formula for the Haar measure. Section \ref{cen} provides the proof of the main result of the paper, namely that the center of the algebra of continuous functions on $\mathrm{SU}_q(2)$ is trivial. This is achieved by examining the commutant of this algebra in the faithful representation studied earlier. These results are used in Section \ref{Haar} to prove that the Haar measure of $\mathrm{SU}_q(2)$ is faithful and its co-unit is continuous (the latter fact is justified in two different ways). Section \ref{gns} is devoted to determining the center of the von Neumann algebra generated by the image of the algebra of continuous functions on $\mathrm{SU}_q(2)$ in the GNS representation for the Haar measure, i.e.~the center of the algebra $\Linf(\mathrm{SU}_q(2))$. Finally in Section \ref{later} we sketch a way to use some of our results and some major results from the literature to obtain an alternative proof of faithfulness of $\bpi$. 

\subsection{Terminology and notation of compact quantum group theory}

We will follow some of the modern texts on quantum groups in declaring a compact quantum group to be an abstract object of the category dual to the category of \cst-algebras related to a unital \cst-algebra with additional structure (see e.g.~\cite{NeshTu} and Section \ref{qg}). In particular for a compact quantum group $\GG$ we write $\C(\GG)$ for the corresponding \cst-algebra which we refer to as the algebra of continuous functions on $\GG$. It is important to note that the actual \emph{set} $\GG$ does not exist.

\subsection{Notation for spectral subspaces and polar decompositions}

In the proofs of some of the results we will employ very useful, yet rather non-standard notation for spectral subspaces of operators introduced in \cite[Section 0]{qef} and put to use e.g.~in \cite[Section 3.1]{so3}. Let $\cH$ be a Hilbert space and $T$ be a normal operator on $\cH$. Let $f$ be a function on the spectrum $\Sp{T}$ of $T$ with values in $\{\text{true},\text{false}\}$. Our notation will be to write $\cH(f(T))$ for the spectral subspace for $T$ corresponding to the subset
\begin{equation}\label{subset}
\bigl\{\lambda\in\Sp{T}\st{f}(\lambda)=\text{true}\bigr\}.
\end{equation}
(we assume that the function $f$ is such that \eqref{subset} is measurable). This notation allows us to write e.g.
\begin{itemize}
\item $\cH(T=\lambda_0)$ for the spectral subspace for $T$ corresponding to $\{\lambda_0\}$,
\item $\cH(|T|>\eps)$ for the spectral subspace for $|T|$ corresponding to $]\eps,+\infty[$, 
\item $\cH(T\neq\lambda_0)$ for the spectral subspace for $T$ corresponding to $\Sp{T}\setminus\{\lambda_0\}$, i.e.~the orthogonal complement of $\cH(T=\lambda_0)$
\end{itemize}
and many other similar expressions. Note that the Hilbert space on which the operator acts is explicitly included in the notation. For example, if an operator $S$ acts on a Hilbert space $\cK$ then we accordingly use notation of the form $\cK(f(S))$, where $f$ is again a $\{\text{true},\text{false}\}$-valued function. Apart from $\cH(\,\dotsm)$ we will also use the symbol $\chi(\,\dotsm)$ to denote the projection onto the corresponding spectral subspace. 

We will also use the following convention for polar decompositions: if, as above, $T$ is a bounded operator on $\cH$, we write $|T|$ for the operator $\sqrt{T^*T}$ and the partial isometry entering the polar decomposition of $T$ will be denoted by $\Phase{T}$. Thus
\[
T=(\Phase{T})|T|
\]
will always denote the polar decomposition of $T$.

\section{The algebra of functions on the quantum group $\mathrm{SU}_q(2)$}\label{Asect}

Let $\sA$ be the universal \cst-algebra generated by two elements $\alpha$ and $\gamma$ subject to relations
\begin{equation}\label{relSUq2}
\begin{aligned}
\alpha^*\alpha+\gamma^*\gamma&=\I,&&\alpha\gamma=q\gamma\alpha,\\
\alpha\alpha^*+q^2\gamma^*\gamma&=\I,&&\gamma^*\gamma=\gamma\gamma^*.
\end{aligned}
\end{equation}
where $q$ is a parameter in $]-1,1[\,\setminus\{0\}$. We remark that due to the Fuglede-Putnam theorem (\cite[Section 12.16]{rudin}) the relation $\alpha\gamma^*=q\gamma^*\alpha$ follows from \eqref{relSUq2} (cf.~\cite[Sections 1.3 \& 3.1]{so3}). It is a matter of simple computation to see that the relations \eqref{relSUq2} are equivalent to unitarity of the matrix
\[
\begin{bmatrix}\alpha&-q\gamma^*\\\gamma&\alpha^*\end{bmatrix}.
\]
This fact immediately shows that the universal \cst-algebra generated by $\alpha$ and $\gamma$ with relations \eqref{relSUq2} exists. Indeed any \cst-seminorm on the $*$-algebra generated by symbols $\alpha$ and $\gamma$ subject to \eqref{relSUq2} must be less or equal to $1$ on entries of a unitary matrix. This implies that for any non-commutative polynomial $a$ in $\alpha,\alpha^*,\gamma,\gamma^*$ and $\I$ the quantity
\[
\|a\|=\sup_{\varrho}\bigl\|\varrho(a)\bigr\|
\]
(where the supremum is taken over all $*$-representations of the $*$-algebra generated by $\alpha$ and $\gamma$ on Hilbert spaces) is finite. This is clearly a \cst-seminorm, but due to \cite[Theorem 1.2]{su2} it is in fact a norm. 
In the proof of this result S.L.~Woronowicz introduced a special representation $\bpi$ of the $*$-algebra generated by $\alpha$ and $\gamma$ which was shown to be injective. Just before statement of \cite[Theorem A2.3]{su2} it is mentioned that the representation $\bpi$ is faithful on $\sA$. We will now give a proof of this result. Before proceeding let us mention that instead of giving a direct proof of faithfulness of $\bpi$ one can use a combination of results of \cite{contraction} or \cite{lance} and \cite{bmt} together with our results to arrive at the same conclusion (cf.~Section \ref{later}).

\subsection{Faithfulness of $\bpi$}\label{fpi}

The universal property of $\sA$ is the following: for any unital \cst-algebra $\sB$ containing two elements $\alpha_0$ and $\gamma_0$ such that
\begin{equation}\label{rel0}
\begin{aligned}
\alpha_0^*\alpha_0+\gamma_0^*\gamma_0&=\I,&&\alpha_0\gamma_0=q\gamma_0\alpha_0,\\
\alpha_0\alpha_0^*+q^2\gamma_0^*\gamma_0&=\I,&&\gamma_0^*\gamma_0=\gamma_0\gamma_0^*.
\end{aligned}
\end{equation}
there is a unique unital $*$-homomorphism $\varrho\colon\sA\to\sB$ such that
\begin{equation}\label{ag0}
\varrho(\alpha)=\alpha_0\quad\text{and}\quad\varrho(\gamma)=\gamma_0.
\end{equation}
By the Gelfand-Naimark theorem this property is equivalent to a simpler one: for any Hilbert space $\cH_0$ and any pair $(\alpha_0,\gamma_0)$ of operators on $\cH_0$ satisfying \eqref{rel0} there exists a unique representation $\varrho$ of $\sA$ on $\cH_0$ such that \eqref{ag0} holds.

\sloppy
Following \cite[Proof of Theorem 1.2]{su2} will now introduce the representation $\bpi$ mentioned above using this universal property: let $\cH=\ell_2(\ZZ_+\times\ZZ)$ and let $\{e_{n,k}\}_{\substack{n\in\ZZ_+\!\!\!\\k\in\ZZ}}$ be the standard orthonormal basis of $\cH$. Define operators $\balpha$ and $\bgamma$ on $\cH$ by
\begin{equation}\label{dzialanie}
\begin{split}
\balpha{e_{n,k}}&=\sqrt{1-q^{2n}}\,e_{n-1,k},\\
\bgamma{e_{n,k}}&=q^n\,e_{n,k+1}.
\end{split}
\end{equation}
The operators $\balpha$ and $\bgamma$ satisfy the defining relations of $\sA$, so there exists a representation $\bpi$ of $\sA$ on $\cH$ such that
\[
\bpi(\alpha)=\balpha\quad\text{and}\quad\bpi(\gamma)=\bgamma.
\]
For future reference let us note the action of $\balpha^*$ and $\bgamma^*$:
\begin{equation}\label{dzialanie2}
\begin{split}
\balpha^*{e_{n,k}}&=\sqrt{1-q^{2(n+1)}}\,e_{n+1,k},\\
\bgamma^*{e_{n,k}}&=q^n\,e_{n,k-1}.
\end{split}
\end{equation}

Our aim is to show that $\bpi$ is faithful. Let $\cst(\balpha,\bgamma)$ be the smallest \cst-algebra of operators on $\cH$ containing $\balpha$ and $\bgamma$. It is easy to see that $\cst(\balpha,\bgamma)$ is the image of the representation $\bpi$. Presently let us note that $\cH$ is isomorphic to $\ell_2(\ZZ_+)\tens\ell_2(\ZZ)$ with the isomorphism mapping $e_{n,k}$ to $e_n\tens{e_k}$ and under this isomorphism the operators $\balpha$ and $\bgamma$ are transformed to
\begin{equation}\label{balga}
\begin{split}
\balpha&=\bs\sqrt{\I-q^{2\bN}}\tens\I,\\
\bgamma&=q^{\bN}\tens\bu,
\end{split}
\end{equation}
where
\begin{itemize}
\item $\bs$ is the unilateral shift
\begin{equation}\label{bs}
\bs\colon{e_n}\longmapsto\begin{cases}e_{n-1}&n>0,\\0&n=0,\end{cases}
\end{equation}
\item $\bN$ is the unbounded self-adjoint operator of multiplication by the sequence $(1,2,3,\ldots)$:
\begin{equation}\label{bN}
\bN\colon{e_n}\longmapsto{n}e_n,\qquad{n\in\ZZ_+},
\end{equation}
\item $\bu$ is the bilateral shift
\[
\bu\colon{e_k}\longmapsto{e_{k+1}},\qquad{k}\in\ZZ.
\]
\end{itemize}
In the proof of Theorem \ref{faithfulPi} below we will establish a similar decomposition for an arbitrary pair of operators satisfying relations \eqref{relSUq2}. Let $\cT$ be the algebra of operators on $\ell_2(\ZZ_+)$ generated by $\bs$ (the Toeplitz algebra, cf.~\cite[Section V.1]{davidson}) and similarly let $\cU$ be the algebra of operators on $\ell_2(\ZZ)$ generated by $\bu$. Then clearly $\balpha$ and $\bgamma$ belong to $\cT\tens\cU\subset\B(\ell_2(\ZZ_+)\tens\ell_2(\ZZ))$ and consequently $\bpi(\sA)\subset\cT\tens\cU$.

\begin{theorem}\label{faithfulPi}
Let $\cH_0$ be a Hilbert space and $(\alpha_0,\gamma_0)$ a pair of bounded operators on $\cH_0$ satisfying relations \eqref{rel0}. Then there exists a unique unital $*$-homomorphism $\varrho\colon\cst(\balpha,\bgamma)\to\B(\cH_0)$ such that
\[
\varrho(\balpha)=\alpha_0\quad\text{and}\quad\varrho(\bgamma)=\gamma_0.
\]
\end{theorem}

\begin{proof}
We have
\begin{subequations}
\begin{align}
\gamma_0^*\gamma_0&=\I-\alpha_0^*\alpha_0,\label{Sp1}\\
q^2\gamma_0^*\gamma_0&=\I-\alpha_0\alpha_0^*.\label{Sp2}
\end{align}
\end{subequations}
The spectrum of \eqref{Sp1} is $1-\Sp(\alpha_0^*\alpha_0)=\bigl\{1-\lambda\st\lambda\in\Sp(\alpha_0^*\alpha_0)\bigr\}$ while the spectrum of \eqref{Sp2} is  $\bigl\{1-\lambda\st\lambda\in\Sp(\alpha_0\alpha_0^*)\bigr\}$. Now recall that $\Sp(\alpha_0^*\alpha_0)\setminus\{0\}=\Sp(\alpha_0\alpha_0^*)\setminus\{0\}$, so
\[
\begin{split}
\Sp(\gamma_0^*\gamma_0)\setminus\{1\}&=\bigl\{1-\lambda\st\lambda\in\Sp(\alpha_0^*\alpha_0)\setminus\{0\}\bigr\}\\
&=\bigl\{1-\lambda\st\lambda\in\Sp(\alpha_0\alpha_0^*)\setminus\{0\}\bigr\}\\
&=\bigl\{1-\lambda\st\lambda\in\Sp(\alpha_0\alpha_0^*)\bigr\}\setminus\{1\}\\
&=q^2\Sp(\gamma_0^*\gamma_0)\setminus\{1\}.
\end{split}
\]
Finally note that since $\|\gamma_0\|\leq{1}$ (this follows easily from the relation $\alpha_0^*\alpha_0+\gamma_0^*\gamma_0=\I$) and $|q|<1$, the set $q^2\Sp(\gamma_0^*\gamma_0)$ is contained in $[0,q^2]$ and so it does not contain $1$. It follows that
\[
\Sp(\gamma_0^*\gamma_0)\setminus\{1\}=q^2\Sp(\gamma_0^*\gamma_0).
\]
It is now easy to see that either $\Sp(\gamma_0^*\gamma_0)=\{0\}$ (in which case $\gamma_0=0$) or
\[
\Sp(\gamma_0^*\gamma_0)=\{0\}\cup\bigl\{q^{2n}\st{n}\in\ZZ_+\bigr\}.
\]

In the former case the relations on $\alpha_0$ and $\gamma_0$ mean simply that $\alpha_0$ is unitary. In the latter we find that the operator $|\gamma_0|$ has discrete spectrum and the space $\cH_0$ can be decomposed into eigenspaces of $|\gamma_0|$:
\[
\cH_0=(\ker\gamma_0)\oplus\biggl(\bigoplus_{n=0}^\infty\cH_0(|\gamma_0|=|q|^n)\biggr).
\]

The operator $\Phase{\gamma_0}$ is zero on $\ker\gamma_0$ and is unitary on $\bigoplus\limits_{n=0}^\infty\cH_0(|\gamma_0|=|q|^n)$. Moreover, since $\Phase{\gamma_0}$ commutes with $|\gamma_0|$ we see that $\Phase{\gamma_0}$ must map each subspace $\cH_0(|\gamma_0|=|q|^n)$ into itself: for $\psi\in\cH_0(|\gamma_0|=|q|^n)$ we have
\[
|\gamma_0|(\Phase{\gamma_0})\psi=(\Phase{\gamma_0})|\gamma_0|\psi=|q|^n(\Phase{\gamma_0})\psi.
\]

\sloppy
Next let us analyze the partial isometry $\Phase{\alpha_0}$. Its initial projection $(\Phase{\alpha_0})^*(\Phase{\alpha_0})$ is the projection onto $\cH_0(|\alpha_0|\neq{0})=\cH_0(|\gamma_0|\neq{1})$, while its final projection $(\Phase{\alpha_0})(\Phase{\alpha_0})^*$ is the projection onto the closure of the range of $\alpha_0$. Note that it follows from $\alpha_0\alpha_0^*=\I-q^2|\gamma_0|^2$ that the range of $\alpha_0$ is all of $\cH_0$, so
\begin{equation}\label{PhAlI}
(\Phase{\alpha_0})(\Phase{\alpha_0})^*=\I.
\end{equation}
Since $|\alpha_0|^2=\I-|\gamma_0|^2$, the relation $\alpha_0\alpha_0^*=\I-q^2|\gamma_0|^2$ can be rewritten as
\[
(\Phase{\alpha_0})\bigl(\I-|\gamma_0|^2\bigr)(\Phase{\alpha_0})^*=\I-q^2|\gamma_0|^2
\]
which in view of \eqref{PhAlI} is
\begin{equation}\label{PgP}
(\Phase{\alpha_0})|\gamma_0|(\Phase{\alpha_0})^*=|q||\gamma_0|.
\end{equation}
Now, multiplying \eqref{PgP} from the right by $\Phase{\alpha_0}$ yields
\[
(\Phase{\alpha_0})|\gamma_0|\chi(|\gamma_0|\neq{1})=|q||\gamma_0|(\Phase{\alpha_0}),
\]
so on $\cH_0(|\gamma_0|\neq{1})$ we have
\begin{equation}\label{PggP}
(\Phase{\alpha_0})|\gamma_0|=|q||\gamma_0|(\Phase{\alpha_0}).
\end{equation}
It follows that $\Phase{\alpha_0}$ maps $\cH_0(|\gamma_0|=1)$ to zero and $\cH_0(|\gamma_0|=|q|^n)$ into $\cH_0(|\gamma_0|=|q|^{n-1})$ for $n>0$. Moreover the map
\[
\bigl.\Phase{\alpha_0}\bigr|_{\cH_0(|\gamma_0|=|q|^n)}\colon\cH_0(|\gamma_0|=|q|^n)\longrightarrow\cH_0(|\gamma_0|=|q|^{n-1})
\]
is onto because the range of $\Phase{\alpha_0}$ is all of $\cH_0$.

Now let us rewrite $\alpha_0\gamma_0=q\gamma_0\alpha_0$ in terms of the respective polar decompositions:
\[
(\Phase{\alpha_0})\sqrt{\I-|\gamma_0|^2}(\Phase{\gamma_0})|\gamma_0|=q(\Phase{\gamma_0})|\gamma_0|(\Phase{\alpha_0})\sqrt{\I-|\gamma_0|^2}
\]
or
\[
(\Phase{\alpha_0})(\Phase{\gamma_0})|\gamma_0|\sqrt{\I-|\gamma_0|^2}=q(\Phase{\gamma_0})|\gamma_0|(\Phase{\alpha_0})\sqrt{\I-|\gamma_0|^2}.
\]
On $\cH_0(|\gamma_0|\neq{1})$ the operator $\sqrt{\I-|\gamma_0|^2}$ is invertible, so
\[
(\Phase{\alpha_0})(\Phase{\gamma_0})|\gamma_0|=q(\Phase{\gamma_0})|\gamma_0|(\Phase{\alpha_0})
\]
and using \eqref{PggP} we obtain
\[
(\Phase{\alpha_0})(\Phase{\gamma_0})|\gamma_0|=\sgn(q)(\Phase{\gamma_0})(\Phase{\alpha_0})|\gamma_0|
\]
on $\cH_0(|\gamma_0|\neq{1})$, where $\sgn(q)=+1$ if $q>0$ and $\sgn(q)=-1$ otherwise. It follows that $(\Phase{\alpha_0})(\Phase{\gamma_0})=\sgn(q)(\Phase{\gamma_0})(\Phase{\alpha_0})$ on $\cH_0(|\gamma_0|\neq{1})$, while on $\cH_0(|\gamma_0|=1)$ we have $\Phase{\alpha_0}=0$, so 
\begin{equation}\label{komutowaniesgn}
(\Phase{\alpha_0})(\Phase{\gamma_0})=\sgn(q)(\Phase{\gamma_0})(\Phase{\alpha_0})
\end{equation}
on all of $\cH_0$.

We now see that the pair $(\alpha_0,\gamma_0)$ is specified uniquely by the normal operator $\gamma_0$ and a partial isometry $\Phase{\alpha_0}$ which is zero on $\cH_0(|\gamma_0|=1)$, maps $\cH_0(|\gamma_0|=|q|^n)$ onto $\cH_0(|\gamma_0|=|q|^{n-1})$ for $n>0$ and satisfies \eqref{komutowaniesgn}. As $\Phase{\gamma_0}$ restricts to a unitary map
\[
\cH_0(|\gamma_0|=|q|^n)\longrightarrow\cH_0(|\gamma_0|=|q|^n)
\]
for each $n$, we see that all these spaces are isomorphic (the isomorphism $\cH_0(|\gamma_0|=|q|^n)\to\cH_0(|\gamma_0|=|q|^{n+1})$ being provided by $(\Phase{\alpha_0})^*$) and the action of $\Phase{\gamma_0}$ on each of these spaces is unitarily equivalent to e.g.~the one on $\cH_0(|\gamma_0|=1)$. In particular writing $\cK_0$ for $\cH_0(|\gamma_0|=1)$ and $u_0$ for $\bigl.\Phase{\gamma_0}\bigr|_{\cK_0}$ we have a unitary operator
\[
U\colon\cH_0(|\gamma_0|\neq{0})\longrightarrow\ell_2(\ZZ_+)\tens\cK_0
\]
with
\[
\begin{split}
U\gamma_0U^*&=q^{\bN}\tens{u_0},\\
U\alpha_0U^*&=\bs\sqrt{\I-q^{2\bN}}\tens\I
\end{split}
\]
where $\bs$ and $\bN$ are the operators described by \eqref{bs} and \eqref{bN} respectively.

Now by universal properties of the \cst-algebras $\cT$ and $\cU$ (\cite[Theorem V.2.2]{davidson}) there exist unique representations
\[
\begin{split}
\varrho_0^1&\colon\cT\longrightarrow\B(\ker{\gamma_0}),\\
\varrho_0^2&\colon\cU\longrightarrow\B(\cK_0),\\
\varrho_0^3&\colon\cU\longrightarrow\CC
\end{split}
\]
such that
\[
\begin{split}
\varrho_0^1(\bs)&=\bigl.(\Phase{\alpha_0})\bigr|_{\ker{\gamma_0}},\\
\varrho_0^2(\bu)&=u_0
\end{split}
\]
and $\varrho_0^3(\bu)=1$. Now let $\varrho$ be the restriction of the mapping
\[
\cT\tens\cU\ni{x}\longmapsto\bigl((\varrho_0^1\tens\varrho_0^3)(x),U^*(\id\tens\varrho_0^2)(x)U\bigr)
\in\B(\ker{\gamma_0})\oplus\B(\cH_0(\gamma_0\neq{0}))\subset\B(\cH_0).
\]
to $\bpi(\sA)\subset\cT\tens\cU$. It satisfies
\[
\varrho(\balpha)=\alpha_0\quad\text{and}\quad\varrho(\bgamma)=\gamma_0.
\]

Uniqueness of $\varrho$ is clear, as $\cst(\balpha,\bgamma)$ is the smallest \cst-algebra of operators on $\cH$ containing $\balpha$ and $\bgamma$ and the value of $\varrho$ on these operators is specified.
\end{proof}

Theorem \ref{faithfulPi} immediately implies the following corollary:

\begin{corollary}
The representation $\bpi$ is faithful. In particular $\sA$ is isomorphic to the \cst-algebra generated by the operators $\balpha$ and $\bgamma$.
\end{corollary}

\section{The quantum group $\mathrm{SU}_q(2)$}\label{qg}

In \cite{su2} S.L.~Woronowicz found that the algebra $\sA$ described in Section \ref{Asect} possesses very rich structure. In particular there is a unique $*$-homomorphism $\Delta\colon\sA\to\sA\tens\sA$ called a \emph{comultiplication} such that
\[
\begin{split}
\Delta(\alpha)&=\alpha\tens\alpha-q\gamma^*\tens\gamma,\\
\Delta(\gamma)&=\gamma\tens\alpha+\alpha^*\tens\gamma
\end{split}
\]
(the tensor product $\sA\tens\sA$ is unambiguous because $\sA$ can be shown to be nuclear, cf.~\cite[Appendix A2]{su2}). Moreover one easily sees that $\Delta$ satisfies
\[
(\Delta\tens\id)\comp\Delta=(\id\tens\Delta)\comp\Delta
\]
i.e.~it is \emph{coassociative}. The pair $(\sA,\Delta)$ satisfies the conditions of \cite[Definition 2.1]{cqg}, so that $\sA$ is an algebra of functions on a compact quantum group. This quantum group is denoted by the symbol $\mathrm{SU}_q(2)$. Hence, the algebra $\sA$ is also denoted by the symbol $\C(\mathrm{SU}_q(2))$.

The isomorphism $\bpi$ of $\C(\mathrm{SU}_q(2))$ onto $\cst(\balpha,\bgamma)$ provides a comultiplication on the latter \cst-algebra. However its existence may be proved directly without knowing that $\bpi$ is faithful. In fact the existence of $\Delta$ on $\cst(\balpha,\bgamma)$ together with several other results can be used to \emph{prove} that $\bpi$ is a faithful representation (see Section \ref{later}). 

The $*$-algebra generated by $\alpha$ and $\gamma$ with appropriate restriction of $\Delta$ is a Hopf $*$-algebra and we denote it by $\Pol(\mathrm{SU}_q(2))$. A convenient basis of $\Pol(\mathrm{SU}_q(2))$ is given by the set $\{a^{k,m,n}\}_{k\in\ZZ,\:m,n\in\ZZ_+}$, where
\begin{equation}\label{baza}
a^{k,m,n}=
\begin{cases}
\alpha^k\gamma^m{\gamma^*}^n&k\geq{0},\\
{\alpha^*}^{-k}\gamma^m{\gamma^*}^n&k<0
\end{cases}
\end{equation}
(\cite[Theorem 1.2]{su2}). For future convenience we will denote the set $\ZZ\times\ZZ_+\times\ZZ_+$ labeling the basis by $\Gamma$.

As any other compact quantum group, the quantum group $\mathrm{SU}_q(2)$ possesses the \emph{Haar measure} which is the unique state $\bh$ on $\C(\mathrm{SU}_q(2))$ which satisfies
\[
(\bh\tens\id)\Delta(a)=\bh(a)\I=(\id\tens\bh)\Delta(a),\qquad{a}\in\C(\mathrm{SU}_q(2)).
\]
(see \cite{su2,pseudogroups,cqg}). The state $\bh$ was found by S.L.~Woronowicz who produced an explicit formula
\[
\bh(a)=(1-q^2)\sum_{n=0}^{\infty}q^{2n}\is{e_{n,0}}{a\,e_{n,0}},\qquad{a}\in\C(\mathrm{SU}_q(2)),
\]
where $\C(\mathrm{SU}_q(2))$ is identified with the \cst-algebra of operators on $\ell_2(\ZZ_+\times\ZZ)$ generated by $\balpha$ and $\bgamma$. As noted already in \cite{pseudogroups} the Haar measure of a compact quantum group need not be faithful. However it is faithful for the quantum $\mathrm{SU}(2)$ groups, a fact whose proof we will give in Section \ref{Haar}.

\section{Center of $\C(\mathrm{SU}_q(2))$}\label{cen}

\subsection{Commutant of $\C(\mathrm{SU}_q(2))$}

By results of Section \ref{faithfulPi} we can identify $\C(\mathrm{SU}_q(2))$ with an algebra of operators on the Hilbert space $\ell_2(\ZZ_+\times\ZZ)$ which we will continue to denote by $\cH$. To keep the notation lighter we will denote the set $\ZZ_+\times\ZZ$ by $\Lambda$. 

We need to introduce a measurable and bounded function $f$:
\[
f\colon \Sp (\bgamma ^*\bgamma)=\{0\}\cup\bigl\{q^{2n}\; \big| \; n\in\ZZ_+\bigr\}\rightarrow \CC\colon \begin{cases} 0\mapsto 1, \\ q^{2n} \mapsto (\sgn (q))^n, \qquad{ n\in\ZZ_+},\end{cases}
\]
and an operator $\underline{\Phase \bgamma}=f(\bgamma^*\bgamma)\Phase \bgamma$. It's easy to see that $\underline{\Phase \bgamma}$ acts on $\cH$ as vertical bilateral shift:
\[
(\underline{\Phase \bgamma})e_{n,k}=e_{n,k+1}, \qquad{(n,k)\in\Lambda}.
\]

We begin by describing the commutant $\C(\mathrm{SU}_q(2))'$ of this algebra of operators.

\begin{theorem}\label{comm}
The commutant of $\C(\mathrm{SU}_q(2))$ in $\B(\cH)$ coincides with the von Neumann algebra generated by $\underline{\Phase{\bgamma}}$.
\end{theorem}

\begin{proof}
Let $T\in\C(\mathrm{SU}_q(2))'$. Define matrix elements $\{T^{n,k}_{n',k'}\}_{(n,k),(n',k')\in\Lambda}$ of $T$ by
\begin{equation}\label{matEl}
Te_{n,k}=\sum_{(n',k')\in\Lambda}T^{n,k}_{n',k'}e_{n',k'},\qquad(n,k)\in\Lambda.
\end{equation}
Fixing $(n,k)\in\Lambda$ we compute
\[
\begin{split}
T\bgamma{e_{n,k}}&=Tq^n{e_{n,k+1}}=\sum_{(n',k')\in\Lambda}T^{n,k+1}_{n',k'}q^n{e_{n', k'}}
=\sum_{(n',k')\in\Lambda}T^{n,k+1}_{n',k'+1}q^{n}{e_{n',k'+1}},\\
\bgamma{T}{e_{n,k}}&=\bgamma\sum_{(n',k')\in\Lambda}T^{n,k}_{n', k'}e_{n',k'}
=\sum_{(n',k')\in\Lambda}T^{n,k}_{n',k'}q^{n'}e_{n',k'+1}.
\end{split}
\]
As $T$ commutes with $\bgamma$, we obtain
\begin{equation}\label{jeden}
T^{n,k+1}_{n',k'+1}q^n=T^{n,k}_{n',k'}q^{n'},\qquad(n,k),(n',k')\in\Lambda.
\end{equation}
Similarly the computation
\[
\begin{split}
T\bgamma^*e_{n,k}&=Tq^n{e_{n,k-1}}=\sum_{(n',k')\in\Lambda}T^{n,k-1}_{n', k'}q^n{e_{n',k'}}
=\sum_{(n',k')\in\Lambda}T^{n,k-1}_{n',k'-1}q^{n}e_{n',k'-1},\\
\bgamma^*T{e_{n,k}}&=\bgamma^*\sum_{(n',k')\in\Lambda}T^{n,k}_{n',k'}e_{n',k'}
=\sum_{(n',k')\in\Lambda}T^{n,k}_{n',k'}q^{n'}e_{n', k'-1}.
\end{split}
\]
and the fact that $T$ commutes with $\bgamma^*$ give
\begin{equation}\label{dwa}
T^{n,k-1}_{n',k'-1}q^n=T^{n,k}_{n',k'}q^{n'},\qquad(n,k),(n',k')\in\Lambda.
\end{equation}

Now take $(n,k),(n',k')\in\Lambda$. Combining \eqref{jeden} and \eqref{dwa} we obtain
\[
T^{n,k}_{n',k'}\overset{\eqref{jeden}}{=}q^{n-n'}T^{n,k+1}_{n',k'+1}\overset{\eqref{dwa}}{=}q^{2(n-n')}T^{n,k}_{n',k'}.
\]
which implies that
\begin{equation}\label{trzy}
T^{n,k}_{n',k'}=0,\qquad(n,k),(n',k')\in\Lambda,\:n\neq{n'}.
\end{equation}

Using this we compute for $(n,k)\in\Lambda$
\[
\begin{split}
T\balpha^*{e_{n,k}}&=T\sqrt{1-q^{2(n+1)}}\,e_{n+1,k}=\sum_{k'\in\ZZ}T^{n+1,k}_{n+1,k'}\sqrt{1-q^{2(n+1)}}\,e_{n+1,k'},\\
\balpha^*T{e_{n,k}}&=\balpha^*\sum_{k'\in\ZZ}T^{n,k}_{n,k'}e_{n,k'}=\sum_{k'\in\ZZ}T^{n,k}_{n,k'}\sqrt{1-q^{2(n+1)}}\,e_{n+1,k'},
\end{split}
\]
and using this together with \eqref{trzy} and \eqref{dwa} we get
\begin{equation}\label{warunekT}
T^{n,k}_{n',k'}=\delta_{n,n'}T^{0,k}_{0,k'}=\delta_{n,n'}T^{0,k-k'}_{0,0},\qquad(n,k),(n',k')\in\Lambda.
\end{equation}

We are now ready to finish the proof. Denote by $\vN(\underline{\Phase{\bgamma}})$ the von Neumann algebra generated by $\underline{\Phase{\bgamma}}$. The inclusion
\[
\vN(\underline{\Phase{\bgamma}})\subset\C(\mathrm{SU}_q(2))'
\]
is clear, as $\underline{\Phase{\bgamma}}$ commutes with $\balpha$ and $\bgamma$ (cf.~\eqref{balga}).

\sloppy
Now take $T\in\C(\mathrm{SU}_q(2))'$  and $S\in\vN(\underline{\Phase{\bgamma}})'$ with matrix elements $\{S^{n,k}_{n',k'}\}_{(n,k),(n',k')\in\Lambda}$. For $k_0\in\NN$ and $(n,k)\in\Lambda$ we have
\[
\begin{split}
S(\underline{\Phase{\bgamma}})^{k_0}e_{n,k}&=\sum_{(n',k')\in\Lambda}S^{n,k+k_0}_{n',k'}e_{n',k'}=\sum_{(n',k')\in\Lambda}S^{n,k+k_0}_{n',k'+k_0}e_{n',k'+k_0},\\
(\underline{\Phase{\bgamma}})^{k_0}Se_{n,k}&=\sum_{(n',k')\in\Lambda}S^{n,k}_{n',k'}e_{n',k'+k_0},
\end{split}
\]
so as $S$ commutes with $\underline{\Phase{\bgamma}}$, we obtain
\begin{equation}\label{warunekS}
S^{n,k}_{n',k'}=S^{n,k+k_0}_{n',k'+k_0}\qquad(n,k),(n',k')\in\Lambda,\;k_0\in\ZZ.
\end{equation}

Now \eqref{warunekT} and \eqref{warunekS} show that $T$ commutes with $S$:
\[
\begin{split}
\is{e_{l,p}}{TSe_{n,k}}&=\is{e_{l,p}}{T\sum_{(n',k')\in\Lambda}S^{n,k}_{n',k'}e_{n',k'}}\\
&=\is{e_{l,p}}{\sum_{(n',k')\in\Lambda}S^{n,k}_{n',k'}\sum_{k''\in\ZZ}T^{0,k'-k''}_{0,0}e_{n',k''}}\\
&=\sum_{k'\in\ZZ}S^{n,k}_{l,k'}T^{0,k'-p}_{0,0}=\sum_{k'\in\ZZ}S^{n,k}_{l,k+p-k'}T^{0,k-k'}_{0,0}\\
&=\sum_{k'\in\ZZ}S^{n,k'}_{l,p}T^{0,k-k'}_{0,0}
\end{split}
\]
and
\[
\begin{split}
\is{e_{l,p}}{STe_{n,k}}&=\is{e_{l,p}}{S\sum_{k''\in\ZZ}T^{0,k-k''}_{0,0}e_{n,k''}}\\
&=\is{e_{l,p}}{\sum_{k''\in\ZZ}T^{0,k-k''}_{0,0}\sum_{(n',k')\in\Lambda}S^{n,k''}_{n',k'}e_{n',k'}}\\
&=\sum_{k''\in\ZZ}T^{0,k-k''}_{0,0}S^{n,k''}_{l,p}.
\end{split}
\]
\sloppy
Since $S$ is arbitrary in $\vN(\underline{\Phase{\bgamma}})'$, this means that $T\in\vN(\underline{\Phase{\bgamma}})''=\vN(\underline{\Phase{\bgamma}})$, so
\[
\vN(\underline{\Phase{\bgamma}})\supset\C(\mathrm{SU}_q(2))'.
\]
\end{proof}

\begin{remark}
Let us point out that the proof of Theorem \ref{comm} does not depend on the analysis of the commutation relations \eqref{relSUq2} presented in the proof of Theorem \ref{faithfulPi}. Using this analysis one can obtain information on the center of the von Neumann algebra generated by the image of a representation of $\C(\mathrm{SU}_q(2))$ (cf.~Section \ref{gns} and Theorem \ref{centerLinf}).
\end{remark}

\subsection{Center of $\C(\mathrm{SU}_q(2))$}

As $\C(\mathrm{SU}_q(2))$ is a unital algebra, its center is non-zero because it must contain all scalar multiples of the unit $\I$. We will show that there are no other central elements in $\C(\mathrm{SU}_q(2))$ or, in other words, that the center of $\C(\mathrm{SU}_q(2))$ is \emph{trivial}.

In the proof of the next theorem we will identify $\C(\mathrm{SU}_q(2))$ with the \cst-algebra of operators on $\ell_2(\ZZ_+\times\ZZ)$ generated by $\balpha$ and $\bgamma$.

\begin{theorem}\label{centrA}
The center of $\C(\mathrm{SU}_q(2))$ is trivial.
\end{theorem}

\begin{proof}
Let $T$ be a central element of $\C(\mathrm{SU}_q(2))$. As $\Pol(\mathrm{SU}_q(2))$ is dense in $\C(\mathrm{SU}_q(2))$. The element $T$ can be approximated in norm by finite linear combinations of elements of the basis \eqref{baza}:
\[
T=\lim_{\lambda\to\infty}\sum_{(k,m,n)\in\Gamma}C^{\lambda}_{k,m,n}a^{k,m,n}.
\]
Given $\eps>0$ there exists $\lambda_0\in\NN$ such that for any $\lambda\geq\lambda_0$ we have
\[
\biggl\|\sum_{(k,m,n)\in\Gamma}C^{\lambda}_{k,m,n}a^{k,m,n}-T\biggr\|\leq\tfrac{\eps}{2}.
\]
As in the proof of Theorem \ref{comm} we will use matrix elements $\{T^{n,k}_{n',k'}\}_{(n,k),(n',k')\in\Lambda}$ of $T$ as defined by \eqref{matEl}.

Take $\lambda\geq\lambda_0$ and $l\in\ZZ_+$. Using \eqref{warunekT}, \eqref{dzialanie} and \eqref{dzialanie2} we obtain
{\allowdisplaybreaks
\begin{align*}
\bigl(\tfrac{\eps}{2}\bigr)^2
&\geq\biggl\|\biggl(\sum_{(k,m,n)\in\Gamma}C^{\lambda}_{k,m,n}a^{k,m,n}-T\biggr)e_{l,0}\bigg\|^2\\
&=\biggl\|\sum_{\substack{k\in\{0,\dotsc,l\}\\m,n\in\ZZ_+}}C^{\lambda}_{k,m,n}q^{l(m+n)}\prod_{i=0}^{k-1}\sqrt{1-q^{2(l-i)}}\,e_{l-k,m-n}+\\
&\qquad+\sum_{\substack{k\in-\NN\\m,n\in\ZZ_+}}C^{\lambda}_{k,m,n}q^{l(m+n)}\prod_{i=0}^{|k|-1}\sqrt{1-q^{2(l+1+i)}}\,e_{l+|k|,m-n}
-\sum_{k\in\ZZ}T^{0,-k}_{0,0}e_{l,k}\biggr\|^2\\
&=\biggl\|\sum_{\substack{k\in\{1,\dotsc,l\}\\s\in\ZZ}}\biggl(\sum_{\substack{m,n\in\ZZ_+\\m-n=s}}C^{\lambda}_{k,m,n}q^{l(m+n)}
\prod_{i=0}^{k-1}\sqrt{1-q^{2(l-i)}}\biggr)e_{l-k,s}\\
&\qquad+\sum_{\substack{k\in-\NN\\s\in\ZZ}}\biggl(\sum_{\substack{m,n\in\ZZ_+\\m-n=s}}C^{\lambda}_{k,m,n}q^{l(m+n)}
\prod_{i=0}^{|k|-1}\sqrt{1-q^{2(l+1+i)}}\biggr)e_{l+|k|,s}\\
&\qquad+\sum_{s\in\ZZ}\biggl(\sum_{\substack{m,n\in\ZZ_+\\m-n=s}}C^\lambda_{0,m,n}q^{l(m+n)}-T^{0,-s}_{0,0}\biggr)e_{l,s}\biggr\|^2\\
&=\sum_{\substack{k\in\{1,\dotsc,l\}\\s\in\ZZ}}\biggl|\sum_{\substack{m,n\in\ZZ_+\\m-n=s}}C^{\lambda}_{k,m,n}q^{l(m+n)}
\prod_{i=0}^{k-1}\sqrt{1-q^{2(l-i)}}\biggr|^2\\
&\qquad+\sum_{\substack{k\in-\NN\\s\in\ZZ}}\biggl|\sum_{\substack{m,n\in\ZZ_+\\m-n=s}}C^{\lambda}_{k,m,n}q^{l(m+n)}
\prod_{i=0}^{|k|-1}\sqrt{1-q^{2(l+1+i)}}\biggr|^2\\
&\qquad+\sum_{s\in\ZZ}\biggl|\sum_{\substack{m,n\in\ZZ_+\\m-n=s}}C^\lambda_{0,m,n}q^{l(m+n)}-T^{0,-s}_{0,0}\biggr|^2\\
&\geq\sum_{s\in\ZZ}\biggl|\sum_{\substack{m,n\in\ZZ_+\\m-n=s}}C^\lambda_{0,m,n}q^{l(m+n)}-T^{0,-s}_{0,0}\biggr|^2.
\end{align*}
}

Thus
\[
\forall_{\eps >0}\;\exists_{\lambda\in\NN}\;\forall_{l\in\ZZ_+}\;\forall_{s\in\ZZ}\quad
\biggl|\sum_{\substack{m,n\in\ZZ_+\\m-n=s}}C^\lambda_{0,m,n}q^{l(m+n)}-T^{0,-s}_{0,0}\biggr|\leq\tfrac{\eps}{2},
\]
so in particular
\begin{equation}\label{asds}
\forall_{\eps >0}\;\forall_{s\in\ZZ}\;\exists_{\lambda\in\NN}\;\forall_{l\in\ZZ_+}\quad\biggl|
\sum_{\substack{m,n\in\ZZ_+\\m-n=s}}C^\lambda_{0,m,n}q^{l(m+n)}-T^{0,-s}_{0,0}\biggr|\leq\tfrac{\eps}{2}.
\end{equation}

Let us take $s\in\ZZ\setminus\{0\}$, $\eps>0$ and let $\lambda\in\NN$ be determined by \eqref{asds}. Since
\[
\lim_{l\to\infty}\sum\limits_{\substack{m,n\in\ZZ_+\\m-n=s}}C^\lambda_{0,m,n}q^{l(m+n)}=0,
\]
there exists $l\in\NN$ such that
\[
\biggl|\sum\limits_{\substack{m,n\in\ZZ_+\\m-n=s}}C^\lambda_{0,m,n}q^{l(m+n)}\biggr|\leq\tfrac{\eps}{2}
\]
(we can pass to the limit under the sum because only a finite number of its terms are non-zero). Combining this with \eqref{asds} we obtain
\[
|T^{0,-s}_{0,0}|\leq\biggl|T^{0,-s}_{0,0}-\sum_{\substack{m,n\in\ZZ_+\\ m-n=s}}C^\lambda_{0,m,n}q^{l(m+n)}\biggr|
+\biggl|\sum_{\substack{m,n\in\ZZ_+\\m-n=s}}C^\lambda_{0,m,n}q^{l(m+n)}\biggr|\leq\tfrac{\eps}{2}+\tfrac{\eps}{2}=\eps,
\]
so for $s\in\ZZ\setminus\{0\}$ we have $|T^{0,-s}_{0,0}|=0$. Consequently $T=T^{0,0}_{0,0}\I$ which ends the proof.
\end{proof}

An immediate consequence of Theorem \ref{centrA} is the following:

\begin{corollary}\label{centrP}
The center of $\Pol(\mathrm{SU}_q(2))$ is trivial.
\end{corollary}

Let us note that Corollary \ref{centrP} can also be easily proved directly by writing a central element of $\Pol(\mathrm{SU}_q(2))$ as a linear combination of the basis \eqref{baza} and checking the conditions implied by commutation with the generators $\alpha$ and $\gamma$.

\section{Faithfulness of Haar measure and continuity of counit}\label{Haar}

\subsection{Faithfulness of the Haar measure}

\begin{proposition}\label{faithBh}
The Haar measure of the quantum group $\mathrm{SU}_q(2)$ is faithful.
\end{proposition}

\begin{proof}
Recall that for $a\in\C(\mathrm{SU}_q(2))\subset\B(\ell_2(\ZZ_+\times\ZZ))$ we have
\[
\bh(a)=(1-q^2)\sum_{n=0}^{\infty}q^{2n}\is{e_{n,0}}{a\,e_{n,0}}.
\]
Assume now that $\bh(b^*b)=0$ for some $b\in\C(\mathrm{SU}_q(2))$. Then for all $n$ we have $\|b\,e_{n,0}\|=0$. Moreover for any $k\in\ZZ$ we have
\[
\begin{split}
\|b\,e_{n,k}\|^2&=\is{b\,e_{n,k}}{b\,e_{n,k}}\\
&=\is{b(\underline{\Phase{\bgamma}})^ke_{n,0}}{b(\underline{\Phase{\bgamma}})^ke_{n,0}}\\
&=\is{(\underline{\Phase{\bgamma}})^kb\,e_{n,0}}{(\underline{\Phase{\bgamma}})^kb\,e_{n,0}}\\
&=\is{b\,e_{n,0}}{b\,e_{n,0}}=\|b\,e_{n,0}\|^2=0
\end{split}
\]
because $b$ commutes with $\underline{\Phase{\bgamma}}$ (cf.~Theorem \ref{comm}). It follows that $b$ is zero on all elements of the standard basis of $\ell_2(\ZZ_+\times\ZZ)$, so $b=0$.
\end{proof}

\subsection{The counit}

The counit of $\mathrm{SU}_q(2)$ is the $*$-character $\eps$ of $\Pol(\mathrm{SU}_q(2))$ determined uniquely by $\eps(\alpha)=1$ and $\eps(\gamma)=0$. By universal property of the \cst-algebra $\sA$ the counit extends to a character $\sA\to\CC$. By Theorem \ref{faithfulPi} $\eps$ can be extended from the $*$-algebra generated by $\balpha$ and $\bgamma$ to $\cst(\balpha,\bgamma)$. In this section we will provide two different proofs of this fact independent of Theorem \ref{faithfulPi} (faithfulness of $\bpi$). This may be used to give an alternative proof of faithfulness of the representation $\bpi$ (cf.~Section \ref{later}).

\begin{theorem}
There exists a character $\widetilde{\eps}\colon\cst(\balpha,\bgamma)\to\CC$ such that
\begin{equation}\label{epsag}
\widetilde{\eps}(\balpha)=1\quad\text{and}\quad\widetilde{\eps}(\bgamma)=0.
\end{equation}
\end{theorem}

\subsubsection{Direct method}

First note that we can easily determine $\eps$ on elements of the basis \eqref{baza} of $\Pol(\mathrm{SU}_q(2))$:
\[
\eps(a^{k,m,n})=
\begin{cases}
1&m=n=0,\\
0&\text{otherwise}.
\end{cases}
\]
Thus in order to show that the counit extends from $\Pol(\mathrm{SU}_q(2))$ to $\cst(\balpha,\bgamma)$ it will be enough to exhibit a continuous linear functional $\omega$ on $\B(\cH)$ such that $\omega(\balpha^{k}\bgamma^m{\bgamma^*}^n)=\delta_{m,0}\delta_{n,0}=\omega({\balpha^*}^{k}\bgamma^m{\bgamma^*}^n)$ for all $k,m,n\in\ZZ_+$. Indeed, as the counit is multiplicative on $\Pol(\mathrm{SU}_q(2))$ it is easy to see that the restriction $\widetilde{\eps}$ of $\omega$ to $\cst(\balpha,\bgamma)$ is a character which maps $\balpha$ to $1$ and $\bgamma$ to $0$.

Define a sequence $(\omega_L)_{L\in\NN}$ of functionals on $\B(\cH)$:
\[
\omega_L(T)=\tfrac{1}{L+1}\is{\sum_{l=0}^{L}e_{l,0}}{T\sum_{l=0}^{L}e_{l,0}},\qquad{T}\in\B(\cH).
\]
Clearly each $\omega_L$ is positive, and so $\|\omega_L\|=\omega_L(\I)=1$ for all $L$. Standard facts about weak topologies (\cite[Chapter 1]{sz}) show that there exists a subsequence $(\omega_{L_p})_{p\in\NN}$ of $(\omega_L)_{L\in\NN}$ weak${}^*$ convergent to $\omega\in\B(\cH)^*$:
\[
\omega=\text{w}^*\text{-}\!\lim_{p\to\infty}\omega_{L_p}
\]
(moreover $\|\omega\|=1$). We will show that $\widetilde{\eps}=\bigl.\omega\bigr|_{\cst(\balpha,\bgamma)}$ satisfies \eqref{epsag}.

\begin{lemma}\label{lemat}
For $k\in\ZZ\setminus\{0\}$ and $L\in\NN$ define
\[
\Xi_{k,L}=
\begin{cases}
\sum\limits_{l=0}^{L-k}\prod\limits_{i=0}^{k-1}\sqrt{1-q^{2(l+k-i)}}&k>0,\rule[-0.7Ex]{0pt}{5Ex}\\
\sum\limits_{l=0}^{L-|k|}\prod\limits_{i=0}^{|k|-1}\sqrt{1-q^{2(l+1+i)}}&k<0.\rule[-3Ex]{0pt}{8Ex}
\end{cases}
\]
Then for any $k\in\ZZ\setminus\{0\}$ we have $\lim_{L\to\infty}\frac{\Xi_{k,L}}{L+1}=1$.
\end{lemma}

\begin{proof}
Fix $k\in\ZZ\setminus\{0\}$ and $(\theta,\pm)=\begin{cases}(k,-)& k>0,\\(1,+)&k<0.\end{cases}$
The series
\[
\sum_{l=0}^{\infty}\biggl(1-\prod_{i=0}^{|k|-1}\sqrt{1-q^{2(l+\theta\pm{i})}}\biggr)
\]
is convergent to some $G_k\in\RR$. Indeed, using in the first step L'H\^opital's rule we obtain
\[
\begin{split}
\lim_{l\to\infty}&\frac{1-\prod\limits_{i=0}^{|k|-1}\sqrt{1-q^{2(l+1+\theta\pm{i})}}}{1-\prod\limits_{i=0}^{|k|-1}\sqrt{1-q^{2(l+\theta\pm i)}}}\\
&\overset{H}{=}\lim_{l\to\infty}\frac{-\sum\limits_{j=0}^{|k|-1}\biggl(\prod\limits_{i\in\{0,\dots,|k|-1\}\setminus\{j\}}
\sqrt{1-q^{2(l+1+\theta\pm i)}}\biggr)\frac{-q^{2(l+1+\theta\pm{j})}\log(q^2)}{2\sqrt{1-q^{2(l+1+\theta\pm{j})}}}}
{-\sum\limits_{j=0}^{|k|-1}\biggl(\prod\limits_{i\in\{0,\dots,|k|-1\}\setminus\{j\}}\sqrt{1-q^{2(l+\theta\pm{i})}}\biggr)
\frac{-q^{2(l+\theta\pm j)}\log(q^2)}{2\sqrt{1-q^{2(l+\theta\pm{j})}}}}\\
&=\lim_{l\to\infty}\frac{\sum\limits_{j=0}^{|k|-1}\biggl(\prod\limits_{i\in\{0,\dots,|k|-1\}\setminus\{j\}}
\sqrt{1-q^{2(l+1+\theta\pm{i})}}\biggr)\frac{q^{2(1+\theta\pm{j})}}{\sqrt{1-q^{2(l+1+\theta\pm{j})}}}}
{\sum\limits_{j=0}^{|k|-1}\biggl(\prod\limits_{i\in\{0,\dots,|k|-1\}\setminus\{j\}}\sqrt{1-q^{2(l+\theta\pm i)}}\biggr)
\frac{q^{2(\theta\pm{j})}}{\sqrt{1-q^{2(l+\theta\pm{j})}}}}\\
&=\frac{\sum\limits_{j=0}^{|k|-1}q^{2(1+\theta\pm{j})}}{\sum\limits_{j=0}^{|k|-1}q^{2(\theta\pm{j})}}=q^2<1,
\end{split}
\]
so the series is convergent by d'Alembert's ratio test.

Then for any $L\geq|k|$ the numbers
\[
R_{L,k}=G_k-\sum_{l=0}^{L-|k|}\biggl(1-\prod_{i=0}^{|k|-1}\sqrt{1-q^{2(l+\theta\pm i)}}\biggr)
\]
satisfy $\lim\limits_{L\to\infty}R_{L,k}=0$. After simple manipulation we arrive at
\[
\Xi_{k,L}=\sum_{l=0}^{L-|k|}\prod_{i=0}^{|k|-1}\sqrt{1-q^{2(l+\theta\pm{i})}}=L+1-|k|-G^{k}+R^{k}_{L},
\]
which proves the lemma.
\end{proof}

We can now check values of $\omega$ on basic elements: choose $k,m,n\in\ZZ_+$ and $L\in\NN$. We have
\[
\begin{split}
\omega_{L}(\balpha^k\bgamma^m{\bgamma^*}^n)
&=\tfrac{1}{L+1}\is{\sum_{l=0}^{L}e_{l,0}}{\balpha^k\bgamma^m{\bgamma^*}^n\sum_{l'=0}^{L}e_{l',0}}\\
&=\tfrac{1}{L+1}\sum_{l=0}^{L}\sum_{l'=0}^{L}[l'-k\ge0]\is{e_{l,0}}{q^{l'(m+n)}\prod_{i=0}^{k-1}\sqrt{1-q^{2(l'-i)}}e_{l'-k,m-n}}\\
&=[m=n]\tfrac{1}{L+1}\sum_{l=0}^{L}\bigl[k+l\in\{0,\dots,L\}\bigr]q^{2(k+l)n}\prod_{i=0}^{k-1}\sqrt{1-q^{2(k+l-i)}}\\
&=\bigl[(m=n)\wedge(L-k\geq{0})\bigr]\tfrac{1}{L+1}\sum_{l=0}^{L-k}q^{2(k+l)n}\prod_{i=0}^{k-1}\sqrt{1-q^{2(k+l-i)}},
\end{split}
\]
where $[\,\dotsm]$ denotes $1$ if the logical expression in brackets is true and $0$ otherwise. It follows that for $m\neq{n}$
\[
\omega(\balpha^k\bgamma^m{\bgamma^*}^n)=\lim_{p\to\infty}\omega_{L_p}(\balpha^k\bgamma^m{\bgamma^*}^n)=0.
\]
When $m=n>0$ we have
\[
\lim_{p\to\infty}\sum_{l=0}^{L_p-k}q^{2(k+l)n}\prod_{i=0}^{k-1}\sqrt{1-q^{2(k+l-i)}}\leq{q}^{2kn}\sum_{l=0}^{\infty}q^{2ln}=\frac{q^{2kn}}{1-q^{2n}}<\infty
\]
for all $p$, so
\[
\begin{split}
\omega(\balpha^k\bgamma^m{\bgamma^*}^n)
&=\lim_{p\to\infty}\omega_{L_p}(\balpha^k\bgamma^m{\bgamma^*}^n)\\
&=\lim_{p\to\infty}\tfrac{1}{L_p+1}\sum_{l=0}^{L_p-k}q^{2(k+l)n}\prod_{i=0}^{k-1}\sqrt{1-q^{2(k+l-i)}}=0.
\end{split}
\]
Similarly
\begin{equation*}	
\begin{split}
\omega_{L}({\balpha^*}^k&\bgamma^m{\bgamma^*}^n)
=\tfrac{1}{L+1}\is{\sum_{l=0}^{L}e_{l,0}}{{\balpha^*}^k\bgamma^m{\bgamma^*}^n\sum_{l'=0}^{L}e_{l',0}}\\
&=\tfrac{1}{L+1}\sum_{l=0}^{L}\sum_{l'=0}^{L}\is{e_{l,0}}{q^{l'(m+n)}\prod_{i=0}^{k-1}\sqrt{1-q^{2(l'+1+i)}}e_{l'+k,m-n}}\\
&=[m=n]\frac{1}{L+1}\sum_{l=0}^{L}\bigl[-k+l\in\{0,\dots,L\}\bigr]q^{2(-k+l)n}\prod_{i=0}^{k-1}\sqrt{1-q^{2(-k+l+1+i)}}\\
&=\bigl[(m=n)\wedge(L\geq{k})\bigr]\tfrac{1}{L+1}\sum_{l=k}^{L}q^{2(-k+l)n}\prod_{i=0}^{k-1}\sqrt{1-q^{2(-k+l+1+i)}}\\
&=\bigl[(m=n)\wedge(L\geq{k})\bigr]\tfrac{1}{L+1}\sum_{l=0}^{L-k}q^{2ln}\prod_{i=0}^{k-1}\sqrt{1-q^{2(l+1+i)}}.
\end{split}
\end{equation*}
and arguing as before we obtain $\omega({\balpha^*}^k\bgamma^m{\bgamma^*}^n)$ for $(m,n)\neq(0,0)$.

Finally for $m=n=0$ and any $k\in\NN$ we have by Lemma \ref{lemat}
\[
\begin{split}
\omega(\balpha^k)&=\lim_{p\to\infty}\tfrac{1}{L_p+1}\sum_{l=0}^{L_p-k}\prod_{i=0}^{k-1}\sqrt{1-q^{2(l+k-i)}}
=\lim_{p\to\infty}\tfrac{\Xi_{k,L_p}}{L_p+1}=1,\\
\omega({\balpha^*}^k)&=\lim_{p\to\infty}\tfrac{1}{L_p+1}\sum_{l=0}^{L_p-k}\prod_{i=0}^{k-1}\sqrt{1-q^{2(l+1+i)}}
=\lim_{p\to\infty}\frac{\Xi_{-k,L_p}}{L_p+1}=1.
\end{split}
\]
As $\omega(\I)=\|\omega\|=1$ we see that $\omega$ has the same values on basic elements of $\Pol(\mathrm{SU}_q(2))$ (embedded in $\cst(\balpha,\bgamma)\subset\B(\cH)$) as the counit.

\begin{remark}
Consider now functionals $\{\omega_L\}_{L\in\NN}$ restricted to $\cst(\balpha,\bgamma)$. Then the above arguments show in fact that the sequence $(\omega_L)_{L\in\NN}$ converges in the weak${}^*$ topology on $\cst(\balpha,\bgamma)^*$ to $\widetilde{\eps}$. Indeed, assuming that $(\omega_L)_{L\in\NN}$ is not convergent we would choose a subsequence $(\omega_{L_p'})_{p\in\NN}$ all of whose elements lie outside a fixed weak${}^*$ neighborhood $\mathcal{O}$ of $\widetilde{\eps}$. But this subsequence would still have a convergent subsequence $(\omega_{L_{p_s}'})_{s\in\NN}$ which by the same arguments as above can be shown to converge on elements $\{\balpha^k\bgamma^m{\bgamma^*}^{n}\}$ and $\{{\balpha^*}^k\bgamma^m{\bgamma^*}^{n}\}$ to $\delta_{m,0}\delta_{n,0}$ which by density of the span of these elements in $\cst(\balpha,\bgamma)$ contradicts the fact that elements of $(\omega_{L_p'})_{p\in\NN}$ lie outside of  $\mathcal{O}$.
\end{remark}

\subsubsection{Another method}

Instead of exhibiting a concrete sequence of vector functionals converging to an extension of $\eps$ to $\cst(\balpha,\bgamma)$ we can use the structure of the pair $(\balpha,\bgamma)$ described in Section \ref{fpi} and the universal property of the Toeplitz algebra to show existence of $\widetilde{\eps}$. Indeed, as we have
\[
\begin{split}
\balpha&=\bs\sqrt{\I-q^{2\bN}}\tens\I,\\
\bgamma&=q^{\bN}\tens\bu,
\end{split}
\]
(when $\cH$ is naturally identified with $\ell_2(\ZZ_+)\tens\ell_2(\ZZ)$) the \cst-algebra $\cst(\balpha,\bgamma)$ is contained in $\cT\tens\cU$ (cf.~Section \ref{fpi}). By universal properties of $\cT$ and $\cU$ there are characters $\varphi$ and $\psi$ of these algebras such that $\varphi(\bs)=1=\psi(\bu)$. Moreover, noting that
\[
q^{\bN}=\sum_{n=0}^{\infty}q^n{\bs^*}^n(\I-\bs^*\bs)\bs^n
\]
we find that
\[
(\varphi\tens\psi)(\bgamma)=
(\varphi\tens\psi)(q^{\bN}\tens\bu\bigr)=\varphi(q^{\bN})=\varphi\biggl(\sum_{n=0}^{\infty}q^n{\bs^*}^n(\I-\bs^*\bs)\bs^n\biggr)=0
\]
and
\[
(\varphi\tens\psi)(\balpha)=\varphi\bigl(\bs\sqrt{\I-q^{2\bN}}\bigr)=\varphi(\bs)\varphi\bigl(\sqrt{\I-q^{2\bN}}\bigr)=\varphi(\bs)\varphi\bigl(\I)=1.
\]

Clearly restriction of $\varphi\tens\psi$ to $\cst(\balpha,\bgamma)$ is the extension of the counit of $\mathrm{SU}_q(2)$.

\section{The GNS representation for $\bh$ and center of $\Linf(\mathrm{SU}_q(2))$}\label{gns}

We begin this section with a concrete realization of the GNS representation of $\C(\mathrm{SU}_q(2))$ for the Haar measure. Recall that
\[
\bh(a)=(1-q^2)\sum_{n=0}^{\infty}q^{2n}\is{e_{n,0}}{a\,e_{n,0}},\qquad{a}\in\C(\mathrm{SU}_q(2)).
\]
Let $(\cH_{\bh},\Omega_{\bh},\pi_{\bh})$ be the GNS triple for $\bh$. We already know that $\bh$ is faithful on $\C(\mathrm{SU}_q(2))$, so $\pi_{\bh}$ is faithful.

As before we write $\cH$ for the carrier Hilbert space of the distinguished representation $\bpi$ considered in Section \ref{fpi}. Consider $\widehat{\cH}=\bigoplus\limits_{n=0}^\infty\cH$ and a vector
\[
\widehat{\Omega}=
\sqrt{1-q^2}\begin{bmatrix}
q^0e_{0,0}\\
q^1e_{1,0}\\
q^2e_{2,0}\\
\vdots
\end{bmatrix}\in\widehat{\cH}.
\]
We have a representation $\widehat{\bpi}$ of $\C(\mathrm{SU}_q(2))$ on $\widehat{\cH}$:
\[
\widehat{\bpi}(a)=
\begin{bmatrix}
a\\
&a\\
&&a\\
&&&\ddots
\end{bmatrix}.
\]
\sloppy
The subspace $\widehat{\cK}=\bigl\{\widehat{\bpi}(a)\widehat{\Omega}{\st{a}\in\C(\mathrm{SU}_q(2))\bigr\}}^{\rule[.5ex]{2ex}{.1ex}}$ is invariant for the action of $\C(\mathrm{SU}_q(2))$ and the resulting representation of $\C(\mathrm{SU}_q(2))$ is equivalent to $\pi_{\bh}$. The appropriate unitary operator $\cH_{\bh}\to\widehat{\cK}$ is given by
\[
\pi_{\bh}(a)\Omega_{\bh}\longmapsto\widehat{\bpi}(a)\widehat{\Omega},\qquad{a}\in\C(\mathrm{SU}_q(2)).
\]
As an immediate corollary we thus get the following:

\begin{proposition}
The operator $\pi_{\bh}(\bgamma)$ has zero kernel.
\end{proposition}

\begin{proof}
Since $\ker{\bgamma}=\{0\}$, we have $\ker\widehat{\bpi}(\bgamma)=\{0\}$ and $\pi_{\bh}(\bgamma)$ is unitarily equivalent to a restriction of $\widehat{\bpi}(\bgamma)$ to the subspace $\widehat{\cK}$.
\end{proof}

The algebra $\Linf(\mathrm{SU}_q(2))$ is by definition the strong closure of $\pi_{\bh}\bigl(\C(\mathrm{SU}_q(2))\bigr)$ in $\B(\cH_{\bh})$. It is a von Neumann algebra and it is referred to as the algebra of \emph{essentially bounded functions on the quantum group $\mathrm{SU}_q(2)$}.

Recall that in Section \ref{cen} we've introduced a function $f$:
\[
f\colon \Sp (\pi_{\bh}(\bgamma^*\bgamma))=\{0\}\cup\bigl\{q^{2n}\; \big| \; n\in\ZZ_+\bigr\}\rightarrow \CC\colon \begin{cases} 0\mapsto 1, \\ q^{2n} \mapsto (\sgn (q))^n, \qquad{ n\in\ZZ_+}.\end{cases}
\]
Let's define an operator $\underline{\Phase{\pi_{\bh}(\bgamma)}}=f(\pi_{\bh}(\bgamma^*\bgamma))\Phase{\pi_{\bh}(\bgamma)}\in \Linf(\mathrm{SU}_q(2))$.

\begin{theorem}\label{centerLinf}
\sloppy
The center of $\Linf(\mathrm{SU}_q(2))$ is the von Neumann subalgebra of $\Linf(\mathrm{SU}_q(2))$ generated by $\underline{\Phase{\pi_{\bh}(\bgamma)}}$.
\end{theorem}

\begin{proof}
Let $\alpha_0$ and $\gamma_0$ be images of $\alpha$ and $\gamma$ in the representation $\pi_{\bh}$. We know already that $\ker{\gamma_0}=\{0\}$, so by the reasoning presented in the proof of Theorem \ref{faithfulPi} we have
\[
\begin{split}
U\gamma_0U^*&=q^{\bN}\tens{u_0},\\
U\alpha_0U^*&=\bs\sqrt{\I-q^{2\bN}}\tens\I,
\end{split}
\]
where $U$ is a unitary $\cH_{\bh}\to\ell_2(\ZZ_+)\tens\cK_0$ for some Hilbert space $\cK_0$ and a unitary $u_0$ on $\cK_0$. 

Consequently the algebra generated by $\underline{\Phase{\gamma_0}}=U^*(\I\tens{u_0})U$ is contained in the center of the von Neumann algebra generated by $\alpha_0$ and $\gamma_0$. 

On the other hand, if $T\in\B(\ell_2(\ZZ_+)\tens\cK_0)=\B(\ell_2(\ZZ_+))\vtens\B(\cK_0)$ commutes with $U\alpha_0U^*$ and $U\gamma_0U^*$ then it must commute with $q^{2\bN}\tens\I$ and $\bs\sqrt{\I-q^{2\bN}}\tens\I$. However, the \cst-algebra generated by these two operators is $\mathcal{K}\tens\I$ where $\mathcal{K}$ is the algebra of compact operators on $\ell_2(\ZZ_+)$. It follows that $T\in\I\tens\B(\cH)$. 

It follows that the center of $\vN(U\alpha_0U^*,U\gamma_0U^*)$ is the von Neumann algebra generated by $\I\tens{u_0}$ which is obviously contained in $\vN(U\alpha_0U^*,U\gamma_0U^*)$. Consequently the center of $\vN(\alpha_0,\gamma_0)$ is generated by $U^*(\I\tens{u_0})U=\underline{\Phase{\gamma_0}}$.
\end{proof}

\section{Another proof of faithfulness of $\bpi$}\label{later}

In this section we will outline an alternative way to prove faithfulness of the representation $\bpi$ considered in Section \ref{fpi}.

First let us note the following facts
\begin{itemize}
\item $\bpi$ is an epimorphism from the universal \cst-algebra $\sA$ generated by $\alpha$ and $\gamma$ satisfying \eqref{relSUq2} onto $\cst(\balpha,\bgamma)$ acting on $\cH=\ell_2(\ZZ_+\times\ZZ)$;
\item the Haar functional on $\sA$ factorizes through $\bpi$ and a faithful functional 
\[
\bh_0\colon{a}\longmapsto(1-q^2)\sum_{n=0}^{\infty}q^{2n}\is{e_{n,0}}{a\,e_{n,0}}
\]
on $\cst(\balpha,\bgamma)$ (by Proposition \ref{faithBh});
\item there exists a comultiplication on $\cst(\balpha,\bgamma)$ which agrees with the comultiplication on $\sA$ (this follows from results of either \cite{contraction} or \cite{lance} and is rather non-trivial), in particular $\cst(\balpha,\bgamma)$ with this comultiplication carries the structure of an algebra of functions on a compact quantum group $\GG_0$;
\item the functional $\bh_0$ is the Haar measure of this quantum group;
\item the map $\bpi$ is injective on the $*$-algebra $\Pol(\mathrm{SU}_q(2))$ generated inside $\sA$ by $\alpha$ and $\gamma$ (\cite[Theorem 1.2]{su2}), in particular we may view $\cst(\balpha,\bgamma)$ as a completion of $\Pol(\mathrm{SU}_q(2))$ for a quantum group norm (\cite[Definition 7.1]{dkps});
\item the algebra $\cst(\balpha,\bgamma)$ admits a character which is a counit for $\GG_0$.
\end{itemize}

By results of \cite{bmt} a Hopf $*$-algebra which admits a completion with continuous counit and faithful Haar measure admits a unique compact quantum group completion. In particular $\sA$ and $\cst(\balpha,\bgamma)$ must be the same and consequently $\bpi$ must be faithful.

\end{document}